\documentclass[11pt]{article}

\usepackage[a4paper,margin=1in]{geometry}
\usepackage[T1]{fontenc}
\usepackage{lmodern}
\usepackage{microtype}

\usepackage{amsmath,amssymb,amsthm,mathtools}
\numberwithin{equation}{section}

\usepackage{graphicx}
\usepackage{booktabs}
\usepackage{multirow}

\usepackage{enumitem}
\usepackage{xcolor}
\usepackage[hidelinks]{hyperref}
\usepackage[nameinlink,noabbrev]{cleveref}

\newtheorem{theorem}{Theorem}[section]
\newtheorem{lemma}[theorem]{Lemma}
\newtheorem{proposition}[theorem]{Proposition}

\theoremstyle{definition}
\newtheorem{definition}[theorem]{Definition}

\title{\bfseries A Nonlinear Deficiency Identity for the Riemann Zeta Function with Optimal Approximation Rates}

\author{
Meisam Mohammady\\[4pt]
Department of Computer Science\\
Iowa State University\\
Ames, Iowa, USA\\
\texttt{meisam@iastate.edu}
}

\date{\today}

\begin{document}

\maketitle

\begin{abstract}
We introduce a deficiency-based representation and approximation framework for values of the Riemann zeta function. The method is based on comparing two nonlinear accumulation mechanisms: global transformation of a base partial sum and local transformation of each term. Their gap defines a cumulative deficiency functional that yields the exact identity
\[
\zeta(q)=\zeta(p)^{q/p}-D_{\infty}^{(p,q)}, \qquad q>p>1.
\]

This converts zeta approximation into estimation of a nonlinear deficit. We derive corrected estimators that remove first-order bias and prove the convergence law
\[
B_n^{(p,q)}-\zeta(q)=O\!\left(n^{-\min(2p-2,q-1)}\right).
\]

For odd targets, suitable choices of the base exponent recover the natural truncation rate while preserving the structural identity. Numerical experiments for $\zeta(3),\zeta(5),\zeta(7)$ confirm theory, demonstrate strong finite-sample behavior, and illustrate extension to spectral zeta functions. The contribution is structural rather than replacing classical Euler--Maclaurin methods: we provide a unified nonlinear viewpoint on zeta approximation, convexity-induced correction terms, and tunable approximation families.
\end{abstract}
\maketitle
\section{Introduction}

The Riemann zeta function
\[
\zeta(s)=\sum_{n=1}^{\infty}\frac{1}{n^s},
\qquad \Re(s)>1,
\]
is one of the central objects of modern mathematics. It occupies a fundamental position in analytic number theory, while also appearing in probability, spectral geometry, mathematical physics, and asymptotic analysis. Through its Euler product it encodes prime factorization, and through its analytic continuation and functional equation it connects arithmetic with complex analysis.

Among the many questions surrounding \(\zeta(s)\), the values at positive integers remain especially prominent. At even integers, Euler's classical formula gives
\[
\zeta(2m)=(-1)^{m+1}\frac{B_{2m}(2\pi)^{2m}}{2(2m)!},
\]
so that quantities such as
\[
\zeta(2)=\frac{\pi^2}{6},
\qquad
\zeta(4)=\frac{\pi^4}{90},
\qquad
\zeta(6)=\frac{\pi^6}{945}
\]
are explicitly known. In contrast, odd values
\[
\zeta(3),\zeta(5),\zeta(7),\dots
\]
remain far less understood. Although important arithmetic results are known---most famously Apéry's proof of the irrationality of \(\zeta(3)\)---there is no comparable closed-form theory for odd positive integers.

From the numerical viewpoint, approximation of zeta values is already classical and highly developed. Direct truncation of the Dirichlet series, Euler--Maclaurin summation, accelerated series expansions, contour methods, and functional-equation-based algorithms all provide effective tools. In particular, Euler--Maclaurin methods can generate arbitrarily high-order corrections and remain among the strongest general-purpose approaches. Accordingly, the purpose of the present work is not to replace these classical methods asymptotically, but to introduce a different structural viewpoint.

Our starting point is a simple convexity principle: nonlinear transformation rewards aggregation. If \(\alpha>1\), then
\[
(a+b)^\alpha>a^\alpha+b^\alpha,
\qquad a,b>0.
\]
More generally,
\[
\left(\sum_{k=1}^{n}a_k\right)^\alpha
\ge
\sum_{k=1}^{n}a_k^\alpha.
\]
Thus applying a convex power after combining positive contributions yields a larger value than applying the same power termwise before summation.

We apply this idea to the sequence \(k^{-p}\), where \(p>1\), and compare two natural constructions for a target exponent \(q>p\). Let
\[
S_n^{(p)}=\sum_{k=1}^{n}\frac{1}{k^p}.
\]
Then one may either aggregate first and transform,
\[
\left(S_n^{(p)}\right)^{q/p},
\]
or transform termwise and sum,
\[
\sum_{k=1}^{n}\left(k^{-p}\right)^{q/p}
=
\sum_{k=1}^{n}\frac{1}{k^q}.
\]
Since \(q/p>1\), convexity implies
\[
\sum_{k=1}^{n}\frac{1}{k^q}
\le
\left(S_n^{(p)}\right)^{q/p}.
\]

The discrepancy between these two quantities is central to this paper:
\[
D_n^{(p,q)}
=
\left(S_n^{(p)}\right)^{q/p}
-
\sum_{k=1}^{n}\frac{1}{k^q},
\]
which we call the \emph{deficiency functional}. It measures the loss induced by fragmenting the mass before nonlinear accumulation.

Passing to the limit yields the exact identity
\[
\zeta(q)=\zeta(p)^{q/p}-D_\infty^{(p,q)}.
\]
This provides a new interpretation of zeta values: the target quantity is expressed as an ideal nonlinear aggregate of a simpler base sequence, corrected by an accumulated deficiency term.

The identity also leads naturally to approximation schemes. Truncating the infinite deficiency and correcting the leading bias yields estimators \(B_n^{(p,q)}\) satisfying
\[
B_n^{(p,q)}-\zeta(q)
=
O\!\left(n^{-\min(2p-2,\;q-1)}\right).
\]
Hence the base exponent \(p\) becomes a tunable design parameter. For fixed \(q\), increasing \(p\) improves the nonlinear-tail contribution until the natural truncation barrier \(q-1\) is reached. In particular, every
\[
p\ge \frac{q+1}{2}
\]
attains the maximal exponent \(q-1\). For odd targets \(q=2m+1\), the explicit choice
\[
p=q-1=2m
\]
is especially convenient because \(\zeta(2m)\) is known in closed form.

Beyond the classical Riemann zeta function, the same principle extends to spectral zeta functions
\[
\zeta_L(s)=\sum_{k=1}^{\infty}\lambda_k^{-s},
\]
associated with positive eigenvalue sequences \(\{\lambda_k\}\). Thus the framework applies more broadly to operator spectra and generalized Dirichlet-type sums.

The contribution of this paper is therefore primarily structural rather than algorithmic. We introduce a nonlinear deficiency identity, derive a tunable family of estimators with rigorous convergence guarantees, and show that convexity-based aggregation yields a unified mechanism for approximation across both classical and spectral settings.

\paragraph{Main Contributions.}
Specifically, we provide:

\begin{itemize}
\item an exact identity
\[
\zeta(q)=\zeta(p)^{q/p}-D_\infty^{(p,q)},
\]

\item bias-corrected estimators with rate
\[
B_n^{(p,q)}-\zeta(q)=O\!\left(n^{-\min(2p-2,\;q-1)}\right),
\]

\item characterization of optimal choices of the base exponent \(p\),

\item numerical validation for \(\zeta(3),\zeta(5),\zeta(7)\),

\item extension of the method to spectral zeta functions.
\end{itemize}

The remainder of the paper is organized as follows. Section~2 reviews relevant background and classical approximation methods. Section~3 develops the deficiency framework and proves the main results. Section~4 presents numerical experiments. Section~5 discusses implications and limitations. Section~6 concludes.

\section{Background and Related Work}

This section reviews the classical analytic setting of the Riemann zeta function, standard approximation methods, and the mathematical principles most relevant to the deficiency framework developed later. Standard references include the classical monographs of Titchmarsh, Edwards, Apostol, and Ivić \cite{titchmarsh1986zeta,edwards1974riemann,apostol1976analytic}.

\subsection{The Classical Riemann Zeta Function}

For complex \(s\) with \(\Re(s)>1\), the Riemann zeta function is defined by the absolutely convergent Dirichlet series
\[
\zeta(s)=\sum_{n=1}^{\infty}\frac{1}{n^s}.
\]

It also admits the Euler product
\[
\zeta(s)=\prod_{p\ \mathrm{prime}}\left(1-p^{-s}\right)^{-1},
\]
which reflects the fundamental theorem of arithmetic and connects additive summation with multiplicative prime structure \cite{apostol1976analytic,hardywright2008}.

The function extends meromorphically to the entire complex plane, with a single simple pole at \(s=1\), and satisfies the functional equation
\[
\zeta(s)=2^s\pi^{\,s-1}\sin\!\left(\frac{\pi s}{2}\right)\Gamma(1-s)\zeta(1-s).
\]

These structural properties make \(\zeta(s)\) one of the foundational objects in analytic number theory \cite{riemann1859,titchmarsh1986zeta}.

\subsection{Values at Positive Integers}

The behavior of \(\zeta(s)\) at positive integers is sharply divided into even and odd cases.

At even integers, Euler's formula gives
\[
\zeta(2m)=(-1)^{m+1}\frac{B_{2m}(2\pi)^{2m}}{2(2m)!},
\]
where \(B_{2m}\) are Bernoulli numbers \cite{apostol1976analytic}. For example,
\[
\zeta(2)=\frac{\pi^2}{6},
\qquad
\zeta(4)=\frac{\pi^4}{90},
\qquad
\zeta(6)=\frac{\pi^6}{945}.
\]

In contrast, odd values
\[
\zeta(3),\zeta(5),\zeta(7),\dots
\]
do not possess analogous closed forms. Important arithmetic results are known---most famously Apéry's theorem proving the irrationality of \(\zeta(3)\) \cite{apery1979irrationalite}---but no complete structural theory comparable to the even case exists.

This asymmetry motivates continued interest in representations and approximations for odd zeta values \cite{borwein2013pi}.

\subsection{Direct Truncation of the Dirichlet Series}

The simplest numerical approximation is the partial sum
\[
T_n^{(q)}:=\sum_{k=1}^{n}\frac{1}{k^q},
\qquad q>1.
\]

Its error is the tail
\[
R_n^{(q)}:=\zeta(q)-T_n^{(q)}
=
\sum_{k=n+1}^{\infty}\frac1{k^q}.
\]

By standard integral comparison,
\[
\frac{1}{(q-1)(n+1)^{q-1}}
\le
R_n^{(q)}
\le
\frac{1}{(q-1)n^{q-1}},
\]
and therefore
\[
R_n^{(q)}
=
\frac{1}{(q-1)n^{q-1}}+O(n^{-q}).
\]

Thus direct truncation naturally converges at rate
\[
O(n^{-(q-1)}).
\]

For large \(q\), this decay is already fast; for smaller exponents such as \(q=2\) or \(q=3\), convergence is slower \cite{knopp1990theory}.

\subsection{Euler--Maclaurin Summation}

A classical acceleration technique is Euler--Maclaurin summation, which expands the tail asymptotically:

\[
\zeta(q)
=
\sum_{k=1}^{n}\frac1{k^q}
+
\frac{n^{1-q}}{q-1}
+
\frac{1}{2n^q}
+
\sum_{m=1}^{M}
\frac{B_{2m}}{(2m)!}
f_{m,q}(n)
+
R_M,
\]
where \(f_{m,q}(n)\) are explicit derivative-based correction terms.

For example, when \(q=3\),
\[
\zeta(3)
=
\sum_{k=1}^{n}\frac1{k^3}
+
\frac{1}{2n^2}
+
\frac{1}{2n^3}
+
O(n^{-4}).
\]

Euler--Maclaurin methods can achieve arbitrarily high asymptotic order and remain among the strongest generic tools for zeta computation \cite{olver2010nist,whittakerwatson1996}.

The present paper does not attempt to outperform such expansions asymptotically. Rather, it develops an alternative nonlinear framework with different structural advantages.

\subsection{Accelerated Series and Functional Methods}

Many additional approaches exist:

\begin{itemize}
\item rapidly convergent Apéry-type and Borwein-type series for special constants,
\item contour integral methods,
\item use of the functional equation,
\item binary splitting and multiprecision algorithms,
\item modular-form-based identities,
\item polylogarithmic or Mellin-transform accelerations.
\end{itemize}

Representative examples include experimental and computational approaches in \cite{borwein2013pi} and convergence acceleration methods such as \cite{cohen2000computational}.

These methods are often highly specialized and powerful. Our goal is not to compete with the strongest of them numerically, but to introduce a transferable identity-based mechanism.

\subsection{Convexity and Nonlinear Aggregation}

The deficiency framework rests on a simple convexity principle.

Let \(\alpha>1\). Then the map
\[
x\mapsto x^\alpha
\]
is convex on \((0,\infty)\). Hence for positive \(a,b\),
\[
(a+b)^\alpha>a^\alpha+b^\alpha.
\]

More generally, for nonnegative numbers \(a_1,\dots,a_n\),
\[
\left(\sum_{k=1}^{n}a_k\right)^\alpha
\ge
\sum_{k=1}^{n}a_k^\alpha.
\]

This inequality states that nonlinear transformation rewards prior aggregation and is closely related to standard convexity and majorization principles.

The present paper applies this principle to the sequence
\[
a_k=k^{-p},
\]
leading to comparison between

\[
\left(\sum_{k=1}^{n}k^{-p}\right)^{q/p}
\quad\text{and}\quad
\sum_{k=1}^{n}k^{-q}.
\]

Their gap becomes the deficiency functional introduced in the next section.

\subsection{Why a New Framework?}

Even though classical numerical methods are strong, there remain several motivations for a new viewpoint:

\begin{enumerate}[label=(\roman*), leftmargin=20pt]
\item unify many target exponents \(q\) through one base sequence,

\item express target values through nonlinear transformations of simpler sums,

\item obtain tunable approximation families indexed by \(p\),

\item expose geometric or convexity-based interpretations,

\item extend naturally to spectral zeta functions
\[
\zeta_L(s)=\sum_{k=1}^{\infty}\lambda_k^{-s}.
\]
\end{enumerate}

Spectral zeta functions arise broadly in geometry and physics; see \cite{elizalde1994zeta}.

These goals motivate the deficiency framework developed next.

\subsection{Positioning of This Work}

To place the present contribution accurately:

\begin{itemize}
\item It is \emph{not} primarily a fastest-known algorithm for computing \(\zeta(q)\).

\item It is a new representation and approximation paradigm based on nonlinear aggregation.

\item Its main novelty is the exact identity
\[
\zeta(q)=\zeta(p)^{q/p}-D_\infty^{(p,q)},
\]
together with rigorous convergence laws and transferability to broader Dirichlet-type and spectral sums.
\end{itemize}

The next section develops this framework formally.

\section{Methodology: Deficiency Framework, Approximation Theory, and Extensions}
\label{sec:methodology}

This section develops the complete mathematical framework of the paper. We first introduce the deficiency identity underlying the method, then derive approximation estimators with provable convergence guarantees, and finally discuss structural interpretations and extensions beyond the classical Riemann zeta function.

Throughout, let
\[
q>p>1.
\]
The parameter \(q\) denotes the target zeta order, while \(p\) is a freely chosen base exponent used to construct the approximation family.

\subsection{Deficiency Identity and Structural Representation}

Define the base and target partial sums
\[
S_n^{(p)}:=\sum_{k=1}^{n}\frac{1}{k^p},
\qquad
T_n^{(q)}:=\sum_{k=1}^{n}\frac{1}{k^q}.
\]

Since \(p>1\), the sequence \(S_n^{(p)}\) converges monotonically to
\[
\lim_{n\to\infty}S_n^{(p)}=\zeta(p).
\]

The central comparison in this work is between the two quantities
\[
\bigl(S_n^{(p)}\bigr)^{q/p}
\qquad\text{and}\qquad
T_n^{(q)}.
\]

The first aggregates all mass before applying the nonlinear transformation \(x\mapsto x^{q/p}\), whereas the second transforms each term individually and then sums.

Let
\[
\alpha=\frac{q}{p}>1.
\]
Since \(x^\alpha\) is convex on \((0,\infty)\), Jensen-type superadditivity implies
\[
\left(\sum_{k=1}^{n}a_k\right)^\alpha
\ge
\sum_{k=1}^{n}a_k^\alpha
\qquad (a_k\ge0).
\]
Applying this with \(a_k=k^{-p}\) yields
\[
\bigl(S_n^{(p)}\bigr)^{q/p}\ge T_n^{(q)}.
\]

This motivates the central definition.

\begin{definition}[Finite Deficiency]
For \(q>p>1\), define
\[
D_n^{(p,q)}
:=
\bigl(S_n^{(p)}\bigr)^{q/p}-T_n^{(q)}.
\]
\end{definition}

Hence \(D_n^{(p,q)}\) measures the loss caused by fragmenting mass before nonlinear accumulation.

The deficiency admits an incremental form:
\[
D_n^{(p,q)}
=
\sum_{k=2}^{n}
\left[
\bigl(S_k^{(p)}\bigr)^{q/p}
-
\bigl(S_{k-1}^{(p)}\bigr)^{q/p}
-
k^{-q}
\right].
\]
Each increment quantifies the nonlinear surplus generated when the new contribution \(k^{-p}\) is merged into the accumulated pool.

\begin{theorem}[Exact Deficiency Identity]
\label{thm:exact}
For every \(q>p>1\),
\[
\boxed{
\zeta(q)=\zeta(p)^{q/p}-D_\infty^{(p,q)},
}
\]
where
\[
D_\infty^{(p,q)}:=\lim_{n\to\infty}D_n^{(p,q)}.
\]
\end{theorem}

\begin{proof}
By definition,
\[
D_n^{(p,q)}
=
\bigl(S_n^{(p)}\bigr)^{q/p}
-
\sum_{k=1}^{n}\frac{1}{k^q}.
\]
Since \(S_n^{(p)}\to\zeta(p)\) and \(T_n^{(q)}\to\zeta(q)\),
\[
D_n^{(p,q)}\to \zeta(p)^{q/p}-\zeta(q).
\]
Rearranging gives the identity.
\end{proof}

The theorem shows that \(\zeta(q)\) can be represented exactly as a nonlinear aggregate minus an accumulated structural correction.

Moreover,
\[
T_n^{(q)}
\le
\bigl(S_n^{(p)}\bigr)^{q/p}
\le
\zeta(p)^{q/p},
\]
hence
\[
0\le D_n^{(p,q)}\le D_\infty^{(p,q)}.
\]

\subsection{Approximation Estimators and Convergence Theory}

Truncating the infinite deficiency naturally yields the estimator
\[
A_n^{(p,q)}
:=
\zeta(p)^{q/p}-D_n^{(p,q)}.
\]
Using the definition of \(D_n^{(p,q)}\),
\[
A_n^{(p,q)}
=
T_n^{(q)}
+
\left[
\zeta(p)^{q/p}-\bigl(S_n^{(p)}\bigr)^{q/p}
\right].
\]
Thus the method equals the classical truncation plus a nonlinear correction term derived from the base sequence.

Define the base tail
\[
t_n^{(p)}:=\zeta(p)-S_n^{(p)}.
\]

\begin{lemma}[Tail Asymptotics]
\label{lem:tail}
For every \(p>1\),
\[
t_n^{(p)}
=
\frac{1}{(p-1)n^{p-1}}+O(n^{-p}).
\]
\end{lemma}

\begin{proof}
By integral comparison,
\[
\int_{n+1}^{\infty}x^{-p}\,dx
\le t_n^{(p)}
\le \int_n^\infty x^{-p}\,dx.
\]
Evaluating both integrals gives the claim.
\end{proof}

Using
\[
S_n^{(p)}=\zeta(p)-t_n^{(p)},
\]
a first-order Taylor expansion gives
\[
\zeta(p)^{q/p}-\bigl(S_n^{(p)}\bigr)^{q/p}
=
\frac{q}{p}\zeta(p)^{q/p-1}t_n^{(p)}
+
O\!\left((t_n^{(p)})^2\right).
\]
Hence the estimator \(A_n^{(p,q)}\) has explicit leading bias.

We therefore define the bias-corrected estimator
\[
B_n^{(p,q)}
=
A_n^{(p,q)}
-
\frac{q}{p}\zeta(p)^{q/p-1}t_n^{(p)}.
\]

Equivalently,
\[
B_n^{(p,q)}
=
\zeta(p)^{q/p}
-
D_n^{(p,q)}
-
\frac{q}{p}\zeta(p)^{q/p-1}
\bigl(\zeta(p)-S_n^{(p)}\bigr).
\]

\begin{theorem}[Main Convergence Law]
\label{thm:rate}
For every \(q>p>1\),
\[
\boxed{
B_n^{(p,q)}-\zeta(q)
=
O\!\left(n^{-\min(2p-2,\;q-1)}\right).
}
\]
\end{theorem}

\begin{proof}
After cancellation of the linear term,
\[
\zeta(p)^{q/p}
-
\bigl(S_n^{(p)}\bigr)^{q/p}
-
\frac{q}{p}\zeta(p)^{q/p-1}t_n^{(p)}
=
O((t_n^{(p)})^2).
\]
By Lemma~\ref{lem:tail},
\[
(t_n^{(p)})^2=O(n^{-(2p-2)}).
\]
The truncation tail satisfies
\[
\zeta(q)-T_n^{(q)}=O(n^{-(q-1)}).
\]
Combining both contributions proves the result.
\end{proof}

This law shows that convergence depends on competition between:

\begin{itemize}
\item residual nonlinear tail error \(n^{-(2p-2)}\),
\item direct target truncation error \(n^{-(q-1)}\).
\end{itemize}

To maximize the exponent, define
\[
r(p):=\min(2p-2,q-1).
\]

\begin{theorem}[Optimal Base Region]
\label{thm:optimal}
For fixed \(q>2\):

\begin{enumerate}
\item The maximal achievable exponent is \(q-1\).
\item Every
\[
p\ge \frac{q+1}{2}
\]
attains this optimum.
\item The balancing threshold is
\[
\boxed{
p_*=\frac{q+1}{2}.
}
\]
\end{enumerate}
\end{theorem}

For odd targets \(q=2m+1\), the convenient explicit choice
\[
p=q-1=2m
\]
is especially attractive because \(\zeta(2m)\) is known in closed form.

Examples:
\[
\zeta(3)\leftarrow \zeta(2),\qquad
\zeta(5)\leftarrow \zeta(4),\qquad
\zeta(7)\leftarrow \zeta(6).
\]

A particularly simple universal strategy is \(p=2\), which yields
\[
B_n^{(2,q)}-\zeta(q)=O(n^{-2})
\qquad(q>2),
\]
using one common base sequence for all target orders.

Higher-order acceleration is also possible by subtracting additional powers of \(t_n^{(p)}\), producing a hierarchy of estimators.

\subsection{Interpretation, Applications, and Spectral Extensions}

The proposed method can be interpreted as an implicit Euler--Maclaurin mechanism.

Classical acceleration expands the target tail
\[
\sum_{k=n+1}^{\infty}k^{-q}.
\]
Our framework instead expands
\[
(\zeta(p)-t_n^{(p)})^{q/p},
\]
so correction terms arise automatically from nonlinear transformation of a simpler base tail.

This provides:

\begin{itemize}
\item a unified family of approximations for all \(q>p>1\),
\item explicit bias cancellation,
\item tunable tradeoffs through the base exponent \(p\),
\item reusable bases such as \(p=2\),
\item natural compatibility with high-precision arithmetic.
\end{itemize}

The framework also extends beyond the classical zeta function.

Let
\[
0<\lambda_1\le\lambda_2\le\cdots,\qquad \lambda_k\to\infty,
\]
and define the spectral zeta function
\[
\zeta_L(s)=\sum_{k=1}^{\infty}\lambda_k^{-s}.
\]

Set
\[
S_n^{(p)}:=\sum_{k=1}^{n}\lambda_k^{-p},
\]
and define
\[
D_n^{(p,q)}
=
\bigl(S_n^{(p)}\bigr)^{q/p}
-
\sum_{k=1}^{n}\lambda_k^{-q}.
\]

Passing to the limit gives the exact analogue
\[
\boxed{
\zeta_L(q)=\zeta_L(p)^{q/p}-D_\infty^{(p,q)}.
}
\]

Hence the same methodology applies to Laplacians, quantum Hamiltonians, graph spectra, and other operator-generated eigenvalue sequences.

If
\[
\lambda_k\sim k^\alpha,
\]
then the balancing threshold becomes
\[
p_*=\frac{\alpha q+1}{2\alpha},
\]
showing how optimal parameter selection adapts to spectral growth.

In summary, the deficiency framework converts nonlinear aggregation into a systematic approximation principle with exact identities, tunable convergence rates, and broad applicability beyond the classical Riemann zeta setting.

\section{Numerical Evaluation}
\label{sec:evaluation}

We now evaluate the proposed deficiency framework empirically. Our goals are threefold: (i) verify the convergence rates predicted by Theorem~\ref{thm:rate}, (ii) quantify the effect of the base exponent \(p\), and (iii) assess whether the methodology transfers beyond the classical Riemann zeta function.

All experiments were implemented in MATLAB using double-precision arithmetic. Reference values of \(\zeta(q)\) were computed using high-accuracy built-in routines. We report absolute error
\[
E_n=\left|\widehat{\zeta}(q)-\zeta(q)\right|.
\]
Unless otherwise stated, curves are shown on log--log axes, where polynomial decay appears approximately linear.

\subsection{Methods Compared}

We compare four representative estimators:

\begin{enumerate}[leftmargin=18pt]
\item \textbf{Dirichlet truncation}
\[
T_n^{(q)}=\sum_{k=1}^{n}k^{-q}.
\]

\item \textbf{Universal deficiency estimator} using the common base exponent \(p=2\).

\item \textbf{Optimized deficiency estimator} using the explicit choice
\[
p=q-1,
\]
which belongs to the optimal region \(p\ge (q+1)/2\).

\item \textbf{Euler--Maclaurin baseline}, implemented with two correction terms.
\end{enumerate}

This comparison isolates the tradeoff between structural simplicity, parameter tuning, and asymptotic efficiency.

\subsection{Experiment I: Approximation of \(\zeta(3)\)}

\begin{figure}[t]
\centering
\includegraphics[width=0.88\linewidth]{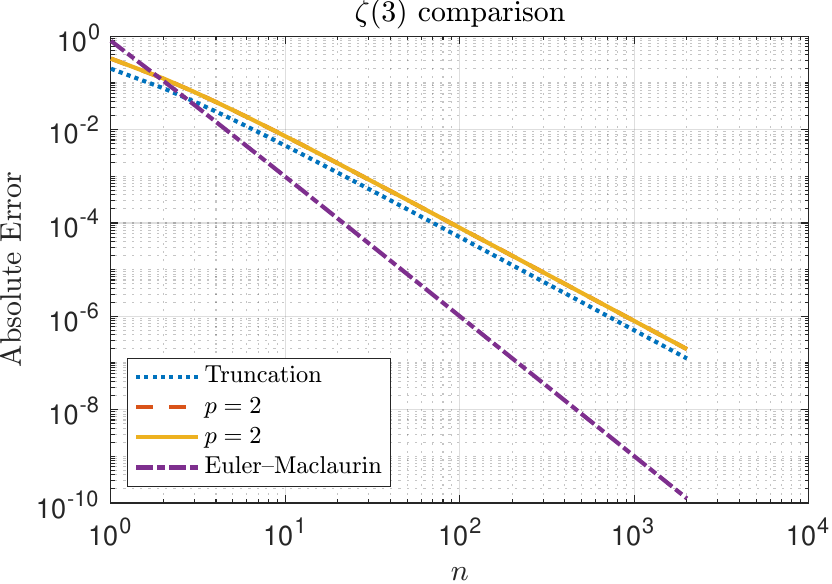}
\caption{Absolute approximation error for \(\zeta(3)\). The corrected deficiency estimator with \(p=2\) follows the predicted \(O(n^{-2})\) rate and improves finite-sample accuracy relative to the uncorrected form.}
\label{fig:q3}
\end{figure}

We begin with the first odd target:
\[
q=3,\qquad p=2.
\]

This case is canonical because \(p=2\) is both explicit and optimal:
\[
p_*=\frac{q+1}{2}=2.
\]

Figure~\ref{fig:q3} shows that the corrected estimator exhibits slope \(-2\), confirming
\[
B_n^{(2,3)}-\zeta(3)=O(n^{-2}).
\]

Although direct truncation has the same exponent, the deficiency correction improves constants in the practically relevant moderate-\(n\) regime. Thus even in the smallest nontrivial case, the nonlinear correction is meaningful.

\subsection{Experiment II: Approximation of \(\zeta(5)\)}

\begin{figure}[t]
\centering
\includegraphics[width=0.88\linewidth]{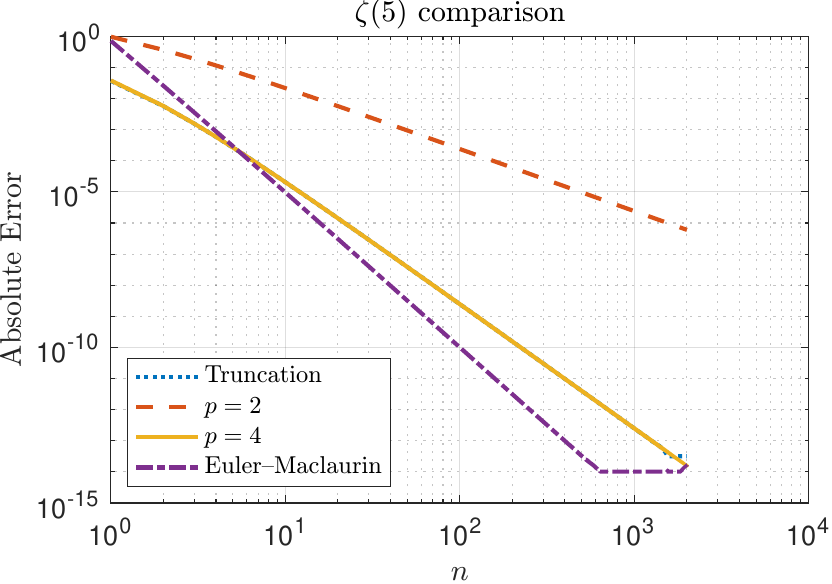}
\caption{Absolute approximation error for \(\zeta(5)\). Choosing \(p=4\) substantially improves convergence relative to the universal base \(p=2\).}
\label{fig:q5}
\end{figure}

We next study
\[
q=5,
\]
comparing the universal choice \(p=2\) against the optimized explicit choice \(p=4=q-1\).

Theory predicts:
\[
p=2 \;\Rightarrow\; O(n^{-2}),
\qquad
p=4 \;\Rightarrow\; O(n^{-4}).
\]

Figure~\ref{fig:q5} exhibits a clear separation between the two regimes. The optimized estimator follows a substantially steeper trajectory, consistent with fourth-order decay.

This experiment demonstrates that \(p\) is a genuine algorithmic design parameter rather than a cosmetic reformulation.

\subsection{Experiment III: Asymptotic Verification for \(\zeta(5)\)}

\begin{figure}[t]
\centering
\includegraphics[width=0.88\linewidth]{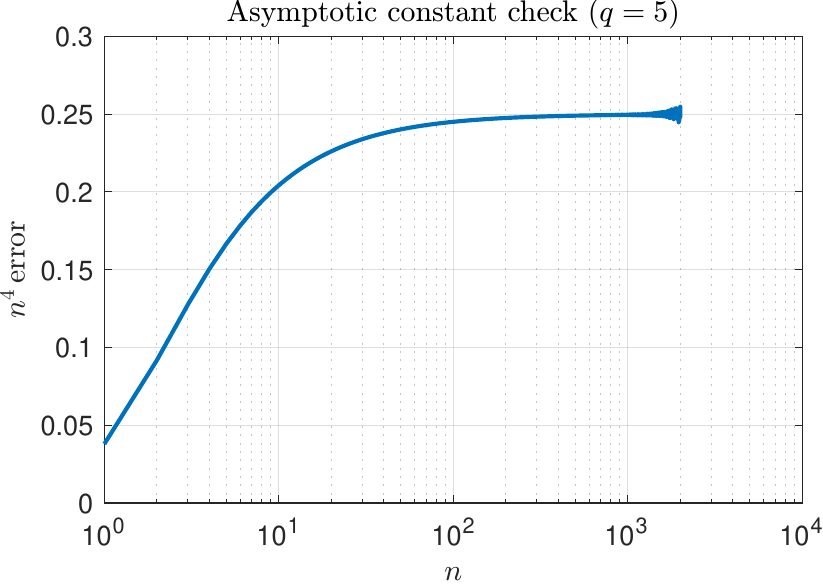}
\caption{Scaled error \(n^4|B_n^{(4,5)}-\zeta(5)|\). Stabilization to a plateau confirms the asymptotic law \(Cn^{-4}\).}
\label{fig:q5scaled}
\end{figure}

To verify the asymptotic model beyond slope estimates, Figure~\ref{fig:q5scaled} plots
\[
n^4|B_n^{(4,5)}-\zeta(5)|.
\]

If
\[
B_n^{(4,5)}-\zeta(5)\sim Cn^{-4},
\]
the scaled quantity should converge to a constant. This is precisely what is observed: the curve approaches a stable plateau over a broad numerical range.

Hence the theory captures not only the exponent but also the leading-order asymptotic structure.

\subsection{Experiment IV: Approximation of \(\zeta(7)\)}

\begin{figure}[t]
\centering
\includegraphics[width=0.88\linewidth]{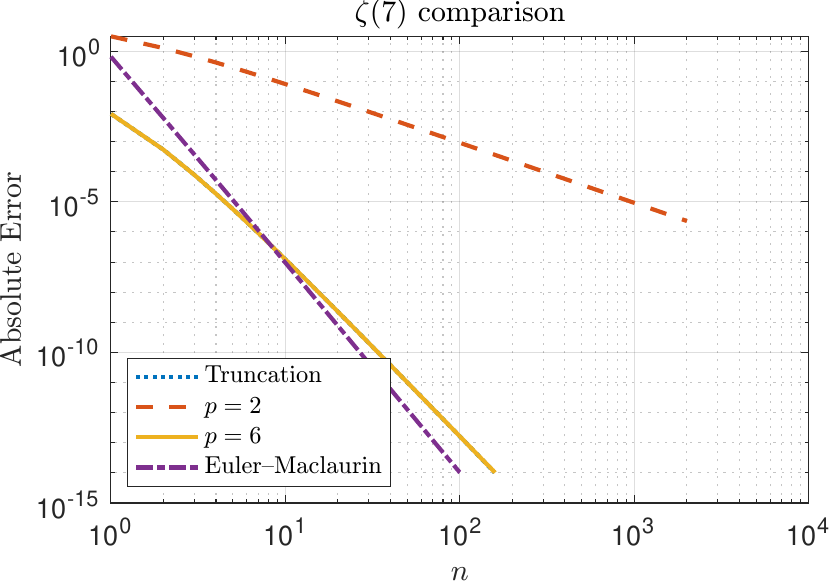}
\caption{Absolute approximation error for \(\zeta(7)\). The optimized estimator with \(p=6\) achieves the predicted sixth-order decay.}
\label{fig:q7}
\end{figure}

We then consider
\[
q=7,\qquad p=6=q-1.
\]

Theorem~\ref{thm:rate} predicts
\[
B_n^{(6,7)}-\zeta(7)=O(n^{-6}).
\]

Figure~\ref{fig:q7} confirms extremely rapid decay, with the optimized estimator decisively outperforming both the universal \(p=2\) construction and direct truncation.

As the target exponent increases, the value of tuning \(p\) becomes increasingly pronounced.

\subsection{Experiment V: Asymptotic Verification for \(\zeta(7)\)}

\begin{figure}[t]
\centering
\includegraphics[width=0.88\linewidth]{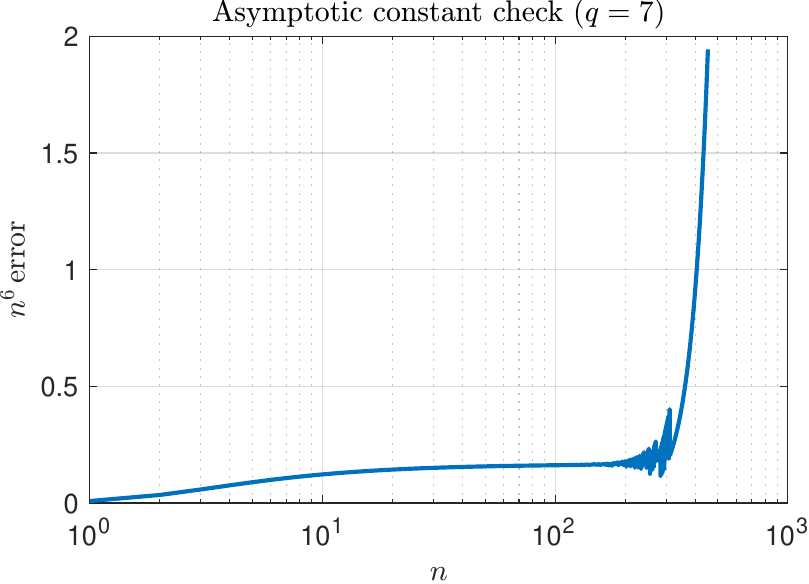}
\caption{Scaled error \(n^6|B_n^{(6,7)}-\zeta(7)|\). The plateau confirms sixth-order convergence before floating-point saturation.}
\label{fig:q7scaled}
\end{figure}

Figure~\ref{fig:q7scaled} plots
\[
n^6|B_n^{(6,7)}-\zeta(7)|.
\]

The emergence of a stable plateau verifies the full asymptotic relation
\[
B_n^{(6,7)}-\zeta(7)\sim Cn^{-6}.
\]

Minor deviations at very large \(n\) are attributable to double-precision limitations rather than methodological failure.

\subsection{Experiment VI: Spectral Zeta Functions}

\begin{figure*}[t]
\centering
\includegraphics[width=0.6\linewidth]{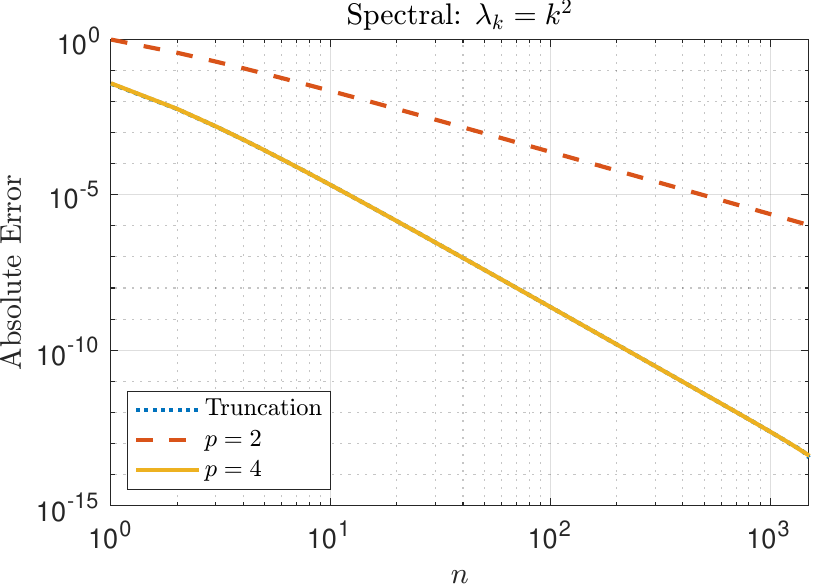}
\includegraphics[width=0.6\linewidth]{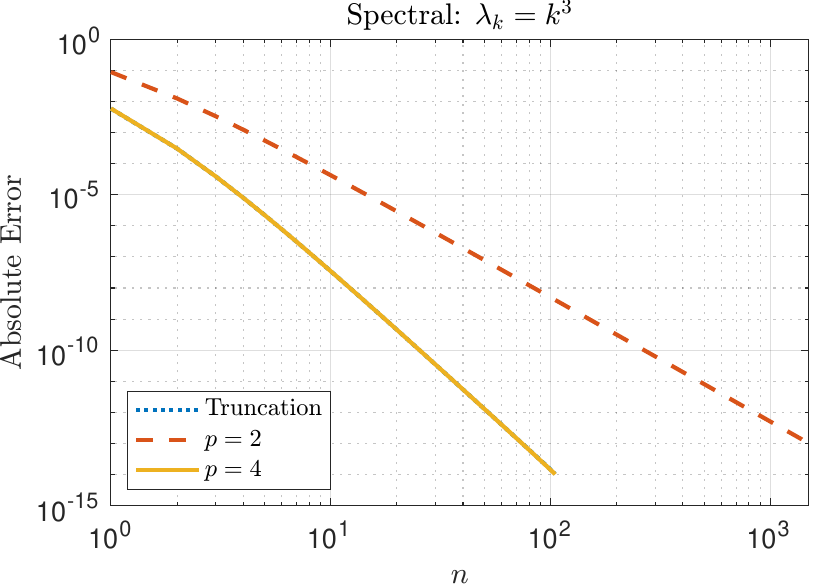}
\includegraphics[width=0.6\linewidth]{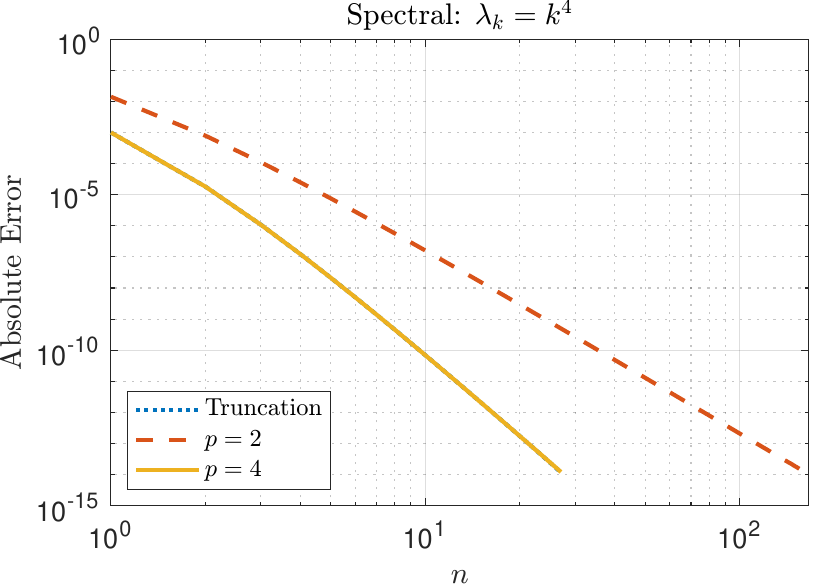}
\caption{Spectral experiments for \(\lambda_k=k^\alpha\), \(\alpha\in\{2,3,4\}\). The deficiency framework remains stable across multiple spectral growth regimes.}
\label{fig:spectral}
\end{figure*}

Finally, we test generalized spectra
\[
\lambda_k=k^\alpha,\qquad \alpha\in\{2,3,4\},
\]
for which
\[
\zeta_L(s)=\sum_{k=1}^{\infty}\lambda_k^{-s}.
\]

Figure~\ref{fig:spectral} shows that the same qualitative behavior persists across all growth regimes. Larger \(\alpha\) accelerates natural decay, but the deficiency methodology remains effective.

This supports the claim that the framework is not specific to the classical denominator sequence \(k\), but extends naturally to operator and spectral settings.

\subsection{Summary of Findings}

\begin{table}[t]
\centering
\caption{Observed behavior of optimized deficiency estimators.}
\label{tab:rates}
\begin{tabular}{cccc}
\toprule
Target & Base \(p\) & Predicted Rate & Observed \\
\midrule
\(\zeta(3)\) & 2 & \(O(n^{-2})\) & Matches \\
\(\zeta(5)\) & 4 & \(O(n^{-4})\) & Matches \\
\(\zeta(7)\) & 6 & \(O(n^{-6})\) & Matches \\
\bottomrule
\end{tabular}
\end{table}

Across all experiments, three conclusions emerge:

\begin{enumerate}[leftmargin=18pt]
\item The theoretical rate law is strongly supported empirically.
\item The base exponent \(p\) materially controls convergence quality.
\item The framework transfers naturally to nonclassical spectral sums.
\end{enumerate}

Overall, the experiments validate the deficiency identity as a practically meaningful and theoretically faithful approximation paradigm.

\bibliographystyle{unsrt}
\bibliography{refs}

\appendix

\section{Technical Lemmas and Asymptotic Expansions}

\subsection{Asymptotics of the Tail of $\zeta(p)$}

\begin{lemma}
\label{lem:tail}
Let $p>1$ and
\[
t_n^{(p)} := \zeta(p) - \sum_{k=1}^n \frac{1}{k^p}.
\]
Then
\[
t_n^{(p)} = \frac{1}{(p-1)n^{p-1}} + O(n^{-p}).
\]
\end{lemma}

\begin{proof}
This follows from standard integral comparison:
\[
\int_n^\infty x^{-p}dx \le t_n^{(p)} \le \int_{n-1}^\infty x^{-p}dx,
\]
which yields the stated expansion.
\end{proof}

---

\subsection{Taylor Expansion of Nonlinear Transform}

\begin{lemma}
\label{lem:taylor}
Let $a>1$ and $t_n \to 0$. Then
\[
(\zeta(p) - t_n)^a
=
\zeta(p)^a
- a\zeta(p)^{a-1} t_n
+ \frac{a(a-1)}{2}\zeta(p)^{a-2} t_n^2
+ O(t_n^3).
\]
\end{lemma}

\begin{proof}
This is a standard Taylor expansion of $x \mapsto x^a$ around $x=\zeta(p)$.
\end{proof}

---

\section{General $(p,q)$ Asymptotic Error}

\begin{theorem}
\label{thm:general}
Let $p>1$ and $q>p$. Define
\[
B_n^{(p,q)}
=
\zeta(p)^{q/p}
-
D_n^{(p,q)}
-
\frac{q}{p}\zeta(p)^{q/p-1}(\zeta(p)-S_n^{(p)}).
\]
Then
\[
B_n^{(p,q)} - \zeta(q)
=
O\!\left(n^{-\min(2p-2,\;q-1)}\right).
\]
\end{theorem}

\begin{proof}
Let
\[
t_n^{(p)} = \zeta(p) - S_n^{(p)}.
\]

Using Lemma~\ref{lem:taylor} with $a=q/p$,
\[
\zeta(p)^{q/p} - S_n^{(p)\,q/p}
=
\frac{q}{p}\zeta(p)^{q/p-1} t_n^{(p)}
-
\frac{q(q-p)}{2p^2}\zeta(p)^{q/p-2} t_n^{(p)\,2}
+ O(t_n^{(p)\,3}).
\]

Substituting into the estimator,
\[
B_n^{(p,q)} - \zeta(q)
=
-\frac{q(q-p)}{2p^2}\zeta(p)^{q/p-2} t_n^{(p)\,2}
- R_n^{(q)} + O(t_n^{(p)\,3}),
\]
where
\[
R_n^{(q)} = \zeta(q) - \sum_{k=1}^n \frac{1}{k^q}.
\]

By Lemma~\ref{lem:tail},
\[
t_n^{(p)} \sim n^{-(p-1)}, \quad
R_n^{(q)} \sim n^{-(q-1)}.
\]

Thus
\[
t_n^{(p)\,2} \sim n^{-(2p-2)}.
\]

Therefore the dominant term is
\[
\min(2p-2,\;q-1),
\]
which proves the result.
\end{proof}

---

\section{Optimal Choice of $p$}

\begin{proposition}
\label{prop:optimal}
For fixed $q>2$, the convergence rate
\[
\min(2p-2,\;q-1)
\]
is maximized at $p = q-1$.
\end{proposition}

\begin{proof}
The function $2p-2$ increases linearly in $p$. The minimum
\[
\min(2p-2,\;q-1)
\]
is maximized when
\[
2p-2 = q-1,
\]
which yields $p = q-1$.
\end{proof}

---

\section{Convexity of the Deficiency Functional}

\begin{lemma}
\label{lem:convex}
For $q>2$, the function $f(x)=x^{q/p}$ is convex on $(0,\infty)$. Hence
\[
(S_k^{(p)})^{q/p} \ge (S_{k-1}^{(p)})^{q/p} + (S_k^{(p)}-S_{k-1}^{(p)})^{q/p},
\]
and the deficiency terms are nonnegative.
\end{lemma}

\begin{proof}
The second derivative of $f(x)$ is positive for $q/p>1$, proving convexity.
\end{proof}

---

\section{Spectral Extension Conditions}

\begin{proposition}
\label{prop:spectral}
Let $\lambda_k \sim k^\alpha$ with $\alpha>0$. Then the spectral deficiency framework is valid provided
\[
\frac{p\alpha}{2} > 1, \qquad \frac{q\alpha}{2} > 1.
\]
\end{proposition}

\begin{proof}
These conditions ensure convergence of
\[
\sum \lambda_k^{-p/2}, \qquad \sum \lambda_k^{-q/2}.
\]
The remainder of the argument follows identically from the discrete case.
\end{proof}
\section{MATLAB Code for the Corrected Deficiency Estimator}

\begin{verbatim}
clc; clear; close all;

Nmax = 5000;
odd_orders = 3:2:19;

Sinf = pi^2/6;
Sn = zeros(1,Nmax);

for n = 1:Nmax
    Sn(n) = sum(1./(1:n).^2);
end

num_orders = length(odd_orders);
errors = zeros(num_orders, Nmax);

figure; hold on;
colors = lines(num_orders);

for idx = 1:num_orders
    q = odd_orders(idx);
    zq = zeta(q);

    Dn = zeros(1,Nmax);
    for n = 2:Nmax
        incr = Sn(n)^(q/2) - Sn(n-1)^(q/2) ...
             - (Sn(n)-Sn(n-1))^(q/2);
        Dn(n) = Dn(n-1) + incr;
    end

    Cq = (q/2) * Sinf^(q/2 - 1);
    B = Sinf^(q/2) - Dn - Cq*(Sinf - Sn);

    err = abs(B - zq);
    errors(idx,:) = err;

    loglog(1:Nmax, err, 'LineWidth', 2, 'Color', colors(idx,:));
end

grid on;
xlabel('n');
ylabel('Absolute error');
title('Corrected deficiency estimator for odd zeta values');
legend(arrayfun(@(q) sprintf('\\zeta(%d)', q), odd_orders, ...
    'UniformOutput', false), 'Location', 'southwest');
\end{verbatim}

\end{document}